\documentclass[12pt]{amsart}
\usepackage{amsmath,amsfonts,amssymb,amsrefs,fullpage,esint} 

\usepackage{amsthm}
\usepackage{tabularx}
\usepackage{tikz}
\usepackage{enumerate}
\newtheorem{thm}{Theorem}
\newtheorem{prop}[thm]{Proposition}

\newtheorem{remark}[thm]{Remark}
\DeclareFontFamily{U}{mathx}{\hyphenchar\font45}
\DeclareFontShape{U}{mathx}{m}{n}{
      <5> <6> <7> <8> <9> <10>
      <10.95> <12> <14.4> <17.28> <20.74> <24.88>
      mathx10
      }{}
\DeclareSymbolFont{mathx}{U}{mathx}{m}{n}
\DeclareFontSubstitution{U}{mathx}{m}{n}
\DeclareMathAccent{\widecheck}{0}{mathx}{"71}
\DeclareMathAccent{\wideparen}{0}{mathx}{"75}
\AtBeginDocument{%
   \def\MR#1{}
}
\numberwithin{equation}{section}
\numberwithin{thm}{section}
\newtheorem*{acknowledgement}{Acknowledgements}

\newcommand{\R}{\mathbb{R}}
\newcommand{\T}{\mathbb{T}}

\newcommand{\Z}{\mathbb{Z}}

\renewcommand{\S}{\mathcal{S}}
\newcommand{\supp}{\operatorname{supp}}
\renewcommand{\angle}{\operatorname{Angle}}

\newcommand{\eps}{\varepsilon}

\newcommand{\ka}{\kappa}

\newcommand{\La}{\Lambda}

\newcommand{\Om}{\Omega}

\newcommand{\ZR}{\mathbb{R}}
\newcommand{\ZZ}{\mathbb{Z}}

\newcommand{\ZT}{\mathbb{T}}

\newcommand{\hichi}{\raisebox{0.7ex}{\(\chi\)}}

\title{$L^p$ weighted Fourier restriction estimates}
\author{Xiumin Du}
\author{Jianhui Li}
\author{Hong Wang}
\author{Ruixiang Zhang}
\begin{document}

\begin{abstract}
We obtain some sharp $L^p$ weighted Fourier restriction estimates of the form $\|Ef\|_{L^p(B^{n+1}(0,R),Hdx)} \lessapprox R^{\beta}\|f\|_2$, where $E$ is the Fourier extension operator over the truncated paraboloid, and $H$ is a weight function on $\mathbb R^{n+1}$ which is $n$-dimensional up to scale $\sqrt R$.
\end{abstract}

\maketitle

\section{Introduction}

Consider the Fourier extension operator over the truncated paraboloid in $\R^{n+1}$:
$$
Ef(x) =\int_{B^n(0,1)} e^{i(x'\cdot \xi + x_{n+1}|\xi|^2)} f(\xi) d\xi, \quad x = (x',x_{n+1}) \in \R^n \times \R.
$$
We say that a weight function $H:\mathbb R^{n+1} \to [0,1]$ is \emph{$\alpha$-dimensional up to scale $S$} if there is a constant $C_H$ such that
$$
\int_{B(x,r)} H(x)dx\leq C_{H} r^{\alpha}, \quad \forall x\in\mathbb R^{n+1}, \forall 1\leq r\leq S. 
$$

Weighted Fourier restriction estimates of the following form have been studied extensively (see for example \cites{DGOWWZ21, DZ19}):
\begin{equation} \label{eq:WEE-0}
   \|Ef \|_{L^{p}(B^{n+1}(0,R),Hdx)} \lesssim R^{\eta} \|f\|_2, 
\end{equation}
where $H$ is any weight function on $\mathbb R^{n+1}$ which is $\alpha$-dimensional up to scale $R$, and the implicit constant depends on $C_H$. Such estimates have several applications in PDEs and geometric measure theory, including the size of the divergence set of Schr\"odinger solutions, spherical average Fourier decay rates of fractal measures, and Falconer's distance set problem. 

It is a challenging problem to determine sharp estimates of the form \eqref{eq:WEE-0} for general $p$ and $\alpha$. As a first step towards this problem, in the current paper, we consider a special case of \eqref{eq:WEE-0} where $Ef$ is essentially supported in a $\sqrt{R}$-neighborhood of an $m$-dimensional subspace with $m<n+1$. We remark that the relevant best-known examples against estimates of the form \eqref{eq:WEE-0} in \cites{DKWZ20, du2019upper} are all in this special case. Our argument works for general $\alpha$, but to be explicit, we only focus on the case $\alpha=n$. By locally constant property, the estimate \eqref{eq:WEE-0} in this special case is equivalent to the following:
\begin{equation} \label{eq:WEE-1}
   \|Ef\|_{L^{p}(B^{m}(0,R),Hdx)} \lesssim R^{\eta+\frac{n+1-m}{2}(\frac 12 -\frac 1p)} \|f\|_2, 
\end{equation}
where $f$ is supported in $B^{m-1}(0,1)$, and $H$ is any weight function on $\mathbb R^{m}$ which is $(m-1)$-dimensional up to scale $\sqrt{R}$. Our main result in this paper is an almost sharp estimate of the form \eqref{eq:WEE-1} for a certain range of $p$. 

\begin{thm} \label{thm-main-0}
    Let $m\geq 2$ and $p_m = \left( \frac{1}{2} - \frac{1}{m^3-m} \right)^{-1}=2+\frac{4}{m^3-m-2}$. For any $\varepsilon>0$,  there exists a constant $C_\varepsilon$ such that the following holds for any $R\geq 1$, any $f$ with $\supp f \subset B^{m-1}(0,1)$, and any weight function $H$ on $\mathbb R^{m}$ which is $(m-1)$-dimensional up to scale $\sqrt{R}$:
    \begin{equation}\label{eq:thm-main-0}
        \|Ef \|_{L^{p_m}(B^{m}(0,R),Hdx)} \leq C_\varepsilon R^\varepsilon R^{\frac{2m-1}{2}(\frac{1}{p_m}-\frac{1}{2}) + \frac{m}{2(m+1)}} \|f\|_{2}.
    \end{equation}
\end{thm}

We can also extend Theorem \ref{thm-main-0} to intermediate dimensions. See Section \ref{sec-WPD} for the definition of a function being concentrated in wave packets from $\T_Z(E)$.

\begin{thm} \label{thm-main}
    Let $2\leq m \leq n+1$ and $p_m = \left( \frac{1}{2} - \frac{1}{m^3-m} \right)^{-1}=2+\frac{4}{m^3-m-2}$. For any $\varepsilon>0$,  there exist constants $C_\varepsilon$ and $0< \delta_{\deg} \ll \delta \ll \varepsilon$ such that the following holds for any $R\geq 1$, any $E\geq R^{\delta}$, any transverse complete intersection $Z = Z(P_1,...,P_{n+1-m})$ where $\deg P_i \leq R^{\delta_{\deg}}$, any $f$ concentrated in wave packets from $\T_Z(E)$ with $\supp f \subset B^n(0,1)$, and any weight function $H$ on $\mathbb R^{n+1}$ which is $n$-dimensional up to scale $\sqrt{R}$:
    \begin{equation}\label{eq:thm-main}
        \|Ef \|_{L^{p_m}(B^{n+1}(0,R),Hdx)} \leq C_\varepsilon E^C R^\varepsilon R^{\frac{n+m}{2}(\frac{1}{p_m}-\frac{1}{2}) + \frac{m}{2(m+1)}} \|f\|_{2},
    \end{equation}
    where $C$ is a large dimensional constant.
\end{thm}

Note that if we take $n+1=m$, Theorem \ref{thm-main} gives exactly Theorem \ref{thm-main-0}.

By the definition in Section \ref{sec-WPD}, if $f$ is concentrated in wave packets from $\T_Z(E)$, then $Ef$ is supported in $N_{ER^{1/2}}(Z)\cap B^{n+1}(0,R)$, where $N_{ER^{1/2}}(Z)$ is the $ER^{1/2}$-neighborhood of $Z$. Due to a theorem of Wongkew \cite{Wongkew}, $N_{ER^{1/2}}(Z)\cap B^{n+1}(0,R)$ can be covered by $\lesssim E^{O(1)} R^{O(\delta_{\deg})} R^{m/2} $ many balls of radius $R^{1/2}$. Hence, by H\"older's inequality and the dimensional condition of $H$ at the scale $\sqrt R$, under the assumptions of Theorem \ref{thm-main} we have the following estimate for any $q\leq p_m$:
\begin{equation}\label{eq:thm-main-q}
        \|Ef \|_{L^{q}(B^{n+1}(0,R),Hdx)} \leq C_\varepsilon E^C R^\varepsilon R^{\frac{n+m}{2}(\frac{1}{q}-\frac{1}{2}) + \frac{m}{2(m+1)}} \|f\|_{2}.
    \end{equation}

\begin{remark}
In Section \ref{sec-eg}, we will show that the exponent of $R$ in \eqref{eq:thm-main-q} is sharp up to the $R^\varepsilon$ loss. However, it seems plausible that \eqref{eq:thm-main-q} could hold for a larger range of $q$ when $m\geq 3$. Based on the example in Section \ref{sec-eg}, we conjecture that under the assumptions of Theorem \ref{thm-main} the estimate \eqref{eq:thm-main-q} holds for $q\leq 2+\frac{4}{(m-1)(m+2)}$.
\end{remark}

As an application, by combining Theorem \ref{thm-main} and the polynomial partitioning method of Guth \cite{Gu18}, one can obtain new $L^p$ estimates for the Schr\"odinger maximal function. See \cite{DLschrodinger}.

In Section \ref{sec-WPD}, we setup the wave packet decomposition and define a function being concentrated in wave packets from $\T_Z(E)$. In Section \ref{sec-eg}, we present a sharp example for Theorem \ref{thm-main}. In Section \ref{sec-mip}, we prove the main inductive proposition from which Theorem \ref{thm-main} follows. The proof is similar to that of the fractal $L^2$ estimate in \cite{DZ19}. The main ingredients include the broad-narrow analysis, multilinear refined Strichartz estimates \cite{DGLZ18}, $l^2$-decoupling theorem \cite{BD2015}, and induction on scales.

\subsection*{Notations.} 
Throughout the article, we write $A\lesssim B$ if $A\leq CB$ for some absolute constant $C$; $A\sim B$ if $A\lesssim B$ and $B\lesssim A$; $A\lesssim_\eps B$ if $A\leq C_\eps B$; $A\lessapprox B$ if $A\leq C_\eps R^\eps B$ for any $\eps>0, R>1$.

For a large parameter $R$, ${\rm RapDec}(R)$ denotes those quantities that are bounded by a huge (absolute) negative power of $R$, i.e. ${\rm RapDec}(R) \leq C_N R^{-N}$ for arbitrarily large $N>0$. Such quantities are negligible in our argument.

\begin{acknowledgement}
XD is supported by NSF CAREER DMS-2237349 and Sloan Research Fellowship. JL is supported by AMS-Simons Travel Grant. HW is supported by NSF CAREER DMS-2238818 and NSF DMS-2055544. RZ is supported by NSF CAREER DMS-2143989 and Sloan Research Fellowship. The first, third and fourth authors would like to thank the 2018 Mathematics Research Communities of AMS  that inspired this research.
\end{acknowledgement}

\section{Wave Packet Decomposition} \label{sec-WPD}

We use the same setup as in Section 3 of \cite{Gu18}. Let $f$ be a function on $B^n(0,1)$, we break it up into pieces $f_{\theta,\nu}$ that are essentially localized in both position and frequency. Cover $B^n(0,1)$ by finitely overlapping balls $\theta$ of radius $R^{-1/2}$ and cover $\ZR^n$ by finitely overlapping balls of radius $R^{\frac{1+\delta}{2}}$, centered at $\nu \in R^{\frac{1+\delta}{2}}\ZZ^{n}$. Using partition of unity, we have a decomposition
$$
f=\sum_{(\theta,\nu)\in \ZT} f_{\theta,\nu} + {\rm RapDec}(R)\|f\|_{L^2}\,,
$$
where $f_{\theta,\nu}$ is supported in $\theta$ and has Fourier transform essentially supported in a ball of radius $R^{1/2+\delta}$ around $\nu$. The functions $f_{\theta,\nu}$ are approximately orthogonal. In other words, for any set $\ZT'\subset \ZT$ of pairs $(\theta,\nu)$, we have
$$
\big\|\sum_{(\theta,\nu)\in \ZT'} f_{\theta,\nu}\big\|_{L^2}^2
\sim \sum_{(\theta,\nu)\in \ZT'} \|f_{\theta,\nu}\|_{L^2}^2\,.
$$
For each pair $(\theta,\nu)$, the restriction of $Ef_{\theta,\nu}$ to $B_R$ is essentially supported on a tube $T_{\theta,\nu}$ with radius $R^{1/2+\delta}$ and length $R$, with direction $G(\theta)\in S^{n}$ determined by $\theta$ and location determined by $\nu$, more precisely,
$$
T_{\theta,\nu} :=\left\{(x',x_{n+1}) \in B_R : |x'+2x_{n+1}\omega_\theta -\nu|\leq R^{1/2+\delta}\right\}\,.
$$
Here $\omega_\theta \in B^n(0,1)$ is the center of $\theta$, and 
$$
G(\theta)=\frac{(-2\omega_\theta,1)}{|(-2\omega_\theta,1)|}\,.
$$

We write $Z(P_1,\cdots,P_{n+1-m})$ for the set of common zeros of the polynomials $P_1,\cdots,P_{n+1-m}$. The variety $Z(P_1,\cdots,P_{n+1-m})$ is called a \emph{transverse complete intersection} if 
$$
\nabla P_1(x) \wedge \cdots \wedge \nabla P_{n+1-m}(x) \neq 0 \text{ for all } x\in Z(P_1,\cdots,P_{n+1-m})\,.
$$
Let $Z$ be an algebraic variety and $E$ a positive number. For any $(\theta,\nu) \in\ZT$,
we say that $T_{\theta,\nu}$ is \emph{$E R^{-1/2}$-tangent} to $Z$ if
$$T_{\theta,\nu}\subset N_{E R^{1/2}}Z \cap B_R,\quad and$$
\begin{equation*}
 \text{Angle}(G(\theta),T_zZ)\leq E R^{-1/2}
\end{equation*}
for any non-singular point $z\in N_{2 E R^{1/2}} ( T_{\theta,\nu}) \cap 2B_R \cap Z$.

Let
$$
\ZT_Z (E):=\{(\theta,\nu)\in\ZT\,|\,T_{\theta,\nu} \text{ is $E R^{-1/2}$-tangent to}\, Z\}\,,
$$
and  we say that $f$ is concentrated in wave packets from $\ZT_Z(E)$ if
$$
 \sum_{(\theta,\nu)\notin \ZT_Z(E)} \|f_{\theta,\nu}\|_{L^2} \leq {\rm RapDec}(R)\|f\|_{L^2}.
$$
Since the radius of $T_{\theta, \nu}$ is $R^{1/2 + \delta}$, $R^\delta$ is the smallest interesting value of $E$.

\section{A Sharp Example} \label{sec-eg}
In this section, we show that the exponent of $R$ in \eqref{eq:thm-main-q} is sharp up to the $R^\varepsilon$ loss, and also that \eqref{eq:thm-main-q} fails when $q>2+\frac{4}{(m-1)(m+2)}$. The example we use is the same as the one in \cite{DOKZFalconerDec}. We include it here for completeness. 

Let $c=1/1000$ be a fixed small constant, $0<\ka<1/2$, and $2\leq k\leq n+1$. Denote
$$x=(x_1,\cdots,x_d)=(x',x'',x_{n+1})\in B^{n+1}(0,R)\,,$$ $$\xi=(\xi_1,\cdots,\xi_{n})=(\xi',\xi'')\in B^{n}(0,1)\,,$$
where $$
x'=(x_1,\cdots,x_{n+1-k}), \quad
x''=(x_{n+2-k},\cdots,x_{n}),$$
$$
\xi'=(\xi_1,\cdots,\xi_{n+1-k}), \quad
\xi''=(\xi_{n+2-k},\cdots,\xi_{n}).$$

For simplicity, we denote $B^{n+1}(0,r)$ by $B^{n+1}_r$,  and write the interval $(-r,r)$ as $I_r$. Let $g(\xi)=\hichi_\Om(\xi)$, where the set $\Om$ is defined by
\begin{equation} \label{Om}
    \Om:=\left[B^{n+1-k}_{cR^{-1/2}} \times \left(2\pi R^{-\ka} \ZZ^{k-1}+B^{k-1}_{cR^{-1}}\right)\right]\cap B^{n}_1\,.
\end{equation}

Next, we define a set $\La$ in $B^{n+1}_R$ by \begin{equation}\label{La}
    \La:=\left[B^{n+1-k}_{cR^{1/2}}\times \left(R^{\ka}\ZZ^{k-1}+B^{k-1}_{c}\right)\times\left(\frac{1}{2\pi}R^{2\ka}\ZZ+I_{c}\right)\right] \cap B^{n+1}_R\,,
\end{equation}
and define $H:=\hichi_\La$. See Figure \ref{fig:Lambda} below.
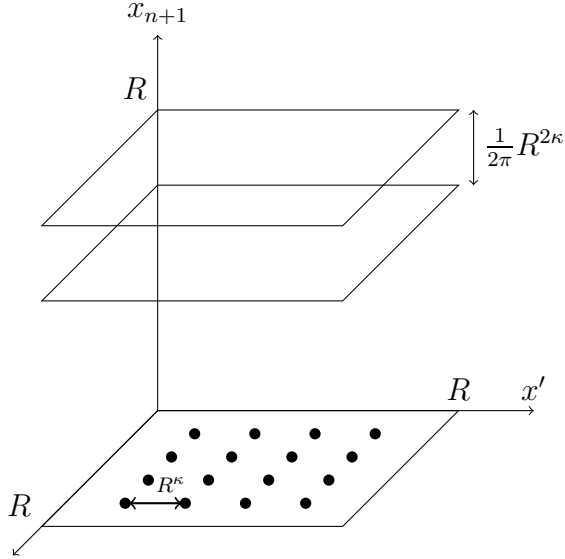
\begin{figure}[ht] 
    \centering
\begin{tikzpicture}
    %axis
    \draw[->] (0,0,0)  -- (5,0,0) node[pos = 0.8, above] {$R$} node[above] {$x'$};
    \draw[->] (0,0,0) -- (0,5,0) node[pos = 0.8, anchor= south east] {$R$} node[above] {$x_{n+1}$};
    \draw[->] (0,0,0) -- (0,0,5) node[pos = 0.8, anchor= south east] {$R$};

    %planes
    \foreach \i in {0,3,4}
    {
    \draw (0,\i,0) -- (0,\i,4) -- (4,\i,4) -- (4,\i,0) -- (0,\i,0);
    }

    %dots  
    \foreach \i in {1,...,4}
        \foreach \j in {1,...,4}
            \foreach \k in {0}
        {
        \node [circle,fill,inner sep=1.5pt] at (\i*4/5,\k,\j*4/5) {};
        }

    \draw[<->] (4.2,4,0) -- (4.2,3,0) node[pos = 0.5, right] {$\frac{1}{2\pi}R^{2\kappa}$};

    \draw[<->, thick] (8/5-0.05,0,4-4/5) -- (4/5+0.05,0,4-4/5) node [pos = 0.2, above] {\tiny $R^{\kappa}$};

    %legend
\end{tikzpicture}

\caption{The set $\Lambda$ intersecting the plane \{$x'' = x''_0$\}, $|x''_0| \leq cR^{1/2}$. Each dot on the $x'$-plane represents a ball of radius $c$. Each plane represents a plank of thickness $c$. }
\label{fig:Lambda}
\end{figure}

From the definition, it follows that
\begin{equation}\label{OmSize}
   |\Om|\sim R^{(\ka-1)(k-1)-(n+1-k)/2}\, ,
\end{equation}
and
\begin{equation}\label{sizeLa}
   |\La|\sim R^{(n+1-k)/2+(1-\ka)(k-1)+1-2\ka}=R^{(n+1+k)/2-\ka(k+1)}\,.
\end{equation}
It is straightforward to check
\begin{equation} \label{xix}
    x'\cdot\xi'+x''\cdot \xi'' +x_{n+1}|\xi'|^2+x_{n+1}|\xi''|^2 \in 2\pi \mathbb Z + (-\frac{1}{100},\frac{1}{100})\,,
\end{equation}
provided that $\xi\in\Om$ and $x\in\La$.

From the constructions, we have the following properties:
\begin{itemize}
    \item Wave packets of $g$ are tangent to a $k$-dimensional subspace;
    \item By \eqref{xix}, $|Eg(x)|\sim |\Om|, \forall x\in \La$;
    \item $\|g\|_2=|\Om|^{1/2}$,
    \item By choosing $\ka:=\frac{1}{2(k+1)}$, one can straightforwardly verify the ball condition at all scales up to $R^{1/2}$:
$$
\int_{Q} H(x)dx =| Q \cap \La | \lesssim r^n, \quad \forall r\text{-ball } Q, \forall 1\leq r\leq R^{1/2};
$$
and also at all scales up to $R$ in the case that $k\leq n$.
\end{itemize}

By a direct calculation we get
$$\frac{\|Eg\|_{L^q(B^{n+1}(0,R);Hdx)}}{\|g\|_2} \gtrsim |\Om|^{1/2} |\La|^{1/q} \sim R^{\frac{n+k}{2}(\frac{1}{q}-\frac{1}{2}) + \frac{k}{2(k+1)}}.
$$

Therefore, if there holds the estimate 
$$
 \|Ef \|_{L^{q}(B^{n+1}(0,R),Hdx)} \lessapprox R^\beta \|f\|_{2}
$$
under the assumptions of Theorem \ref{thm-main}, then 
$$
\beta \geq \max_{2\leq k\leq m} \frac{n+k}{2}(\frac{1}{q}-\frac{1}{2}) + \frac{k}{2(k+1)}.
$$
By taking $k=m$, we see that the exponent $\frac{n+m}{2}(\frac{1}{q}-\frac{1}{2}) + \frac{m}{2(m+1)}$ of $R$ in \eqref{eq:thm-main-q} is sharp up to the $R^\varepsilon$ loss. On the other hand, by taking $k=m-1$ we see that if \eqref{eq:thm-main-q} holds, then we must have 
\[
\frac{n+m}{2}(\frac{1}{q}-\frac{1}{2}) + \frac{m}{2(m+1)} \geq \frac{n+m-1}{2}(\frac{1}{q}-\frac{1}{2}) + \frac{m-1}{2m}\,,
\]
i.e. $q\leq 2+\frac{4}{(m-1)(m+2)}$.

\section{Main Inductive Proposition} \label{sec-mip}
By a similar argument as in \cite{DZ19}, Theorem \ref{thm-main} is a result of the following inductive proposition. We say that the quantities in a collection are dyadically constant if all the quantities are in the same interval of the form $[2^j,2^{j+1}]$, where $j$ is an integer.

\begin{prop}\label{prop:induct}
    Let $2\leq m \leq n+1$ and $p = \frac{2m}{m-2}$ ($p=\infty$ when m=2). For any $\varepsilon>0$,  there exist constants $C_\varepsilon$ and $0< \delta_{\deg} \ll \delta \ll \varepsilon$ such that the following holds for all $R\geq 1$, all $E\geq R^{\delta}$, all transverse complete intersection $Z = Z(P_1,...,P_{n+1-m})$ where $\deg P_i \leq R^{\delta_{\deg}}$, and all $f$ concentrated in wave packets from $\T_Z(E)$ with $\supp f \subset B^n(0,1)$. Let $K = R^\delta$. Suppose that $Y = \cup_{k=1}^M B_k$ is a union of lattice $K^2$-cubes in $ N _{ER^{1/2}}(Z) \cap B^{n+1}(0,R)$.

    Let $\gamma$ be given by
    \begin{equation}\label{eq:prop:induct:ndimassu}
        \gamma: =  \max_{\substack{B^{n+1}(x',r) \subset B^{n+1}(0,R)\\x' \in B^{n+1}(0,R) , K^2\leq r \leq \sqrt{R} }} \frac{\# \{B_k: B_k \subset B(x',r)\}}{r^n}.
    \end{equation}

    Suppose further that 
    $$\|Ef\|_{L^p(B_k)} \text{ is dyadically a constant in }k = 1,2,...,M.$$

    Then
    \begin{equation}\label{eq:prop:induct}
        \|Ef \|_{L^p(Y)} \leq C_\varepsilon E^C R^\varepsilon \left ( \frac{\gamma}{M} \right)^{\frac{1}{p_m} - \frac{1}{p}} R^{\frac{n+m}{2}(\frac{1}{p_m}-\frac{1}{2}) + \frac{m}{2(m+1)}} \|f\|_{2},
    \end{equation}
    where
    $$
    2 < p_m = \left( \frac{1}{2} - \frac{1}{m^3-m} \right)^{-1} < p
    $$
    and $C$ is a large dimensional constant.
\end{prop}

\begin{proof} We induct on the radius $R$. Note that estimate \eqref{eq:prop:induct} is trivial when $R\lesssim 1$. We may assume that Proposition \ref{prop:induct} is true for $R$ replaced by $R_1=R/K^2$.

We start with a broad-narrow analysis following \cites{BG11,Bo12,BD2015,Gu16}. We decompose the frequency space $B^{n}(0,1)$ into disjoint $K^{-1}$-cubes $\tau$. For each physical $K^2$-cube $B$, define the \textit{significant} set
$$
\S(B) : =\left  \{ \tau : \|E f_{\tau}\|_{L^p(B)} \geq \frac{1}{(100K)^{n}} \|Ef\|_{L^p(B)} \right\}.
$$

Since the number of $K^{-1}$-cubes $\tau$ is $\sim K^n$, triangle inequality gives
$$
\left \|\sum_{\tau \in \S(B)} E f_\tau \right\|_{L^p(B)} \sim \|Ef\|_{L^p(B)}.
$$
The above relation explains the name significant set $\S(B)$.

%On any $K^2$-cube $B$ near the $m$-dimensional algebraic variety $Z$, we may assume by a standard pigeonholing argument that the directions of the wave packets from $\T_Z(E)$ passing through $B$ are concentrated near some $m$-dimensional subspaces. See \cite{Gu18}*{Subsection 8.1}.

%The physical $K^2$-cube $B$ is said to be \textit{narrow} if the directions of the wave packets with frequencies on the significant set $\S(B)$ lie near an $(m-1)$-dimensional subspace $V$, i.e.

The physical $K^2$-cube $B$ is said to be \textit{narrow} if there is an $(m-1)$-dimensional subspace $V$ such that for all $\tau \in \S(B)$,
$$
\angle(G(\tau),V) \leq \frac{1}{100nK}.
$$ Otherwise, the physical $K^2$-cube $B$ is said to be \textit{broad}. By the definition of broad cube $B$, there exist a constant $C_n^{br}$ and caps $\tau_1,...,\tau_m \in \S(B)$ such that for any $v_j \in G(\tau_j)$,
$$
|v_1 \wedge v_2 \wedge ... \wedge v_m| \geq C_n^{br} K^{-n}.
$$

Recall that $\|Ef\|_{L^p(B_k)}$ is dyadically a constant for $B_k \subset Y = \cup_{k=1}^M B_k$. To control $\|Ef\|_{L^p(Y)}$, it suffices to estimate $M/2$ many $B_k$'s. If at least a half of the cubes $B_k$ is broad, we will estimate these broad $B_k$ only in Subsection \ref{sec:broad}. If otherwise, we will estimate narrow $B_k$ only in Subsection \ref{sec:narrow}.

\subsection{Broad case} \label{sec:broad}

Define the collection of $m$-tuple of transverse caps:
\begin{equation}\label{eq:m-transverse}
    \Gamma: = \{ (\tau_1,...,\tau_m) : |v_1 \wedge ... \wedge v_m | \geq C_n^{br} K^{-n} \quad \text{for any } v_j \in G(\tau_j)\}.
\end{equation}

For each broad $B$, there exists $(\tau_1(B),...,\tau_m(B)) \in \Gamma$ such that
$$
\|Ef\|_{L^p(B)}^p \lesssim K^{O(1)} \|Ef_{\tau_j(B)}\|_{L^p(B)}^p
$$
for all $1\leq j \leq m$.
Therefore,
\begin{equation}\label{eq:broad1}
    \|Ef\|_{L^p(B)}^p \lesssim K^{O(1)} \prod_{j=1}^m\|Ef_{\tau_j(B)}\|_{L^p(B)}^{p/m}.
\end{equation}

By physical translations and locally constant principle (see for instance \cite{DZ19}), the geometric average and the $L^p$ norm over $K^2$-cube $B$ on the right hand side of \eqref{eq:broad1} can be interchanged at a cost of $K^{O(1)}$ and harmless physical translations:
\begin{equation}\label{eq:broad2}
    \prod_{j=1}^m\|Ef_{\tau_j(B)}\|_{L^p(B)}^{p/m} \lesssim K^{O(1)} \left \|\prod_{j=1}^m |Ef_{\tau_j(B) , v_j(B)}|^{1/m}\right\|_{L^p(B)}^{p},
\end{equation}
where 
$$
v_j(B) \in B(0,K^2) \cap \Z^{n+1}, \quad f_{\tau_j , (v',v_{n+1})}(\xi) = f_{\tau_j}(\xi) e^{i ( v' \cdot \xi  + v_{n+1} |\xi|^2)}.
$$

Since there are only $K^{O(1)}$ many choices in $(\tau_1,...,\tau_m) \in \Gamma$ and $v_1,...,v_m \in B(0,K^2) \cap \Z^{n+1}$, there exists $(\tau_1,...,\tau_m) \in \Gamma$ and $v_1,...,v_m \in B(0,K^2) \cap \Z^{n+1}$ such that $\tau_j = \tau_j(B_k)$ and $v_j=v_j(B_k), j=1,...,m$ for $K^{-C} M$ many broad $B_k \subset Y$. Let $\mathcal{B}$ denote the collection of these broad cubes $B_k$. In the rest of the proof of the broad case, we fix caps $(\tau_1,...,\tau_m)$ and translations $v_1,...,v_{m}$. For brevity, write $f_j = f_{\tau_j,v_j}$. Since $\|Ef\|_{L^p(B_k)}$ is dyadically a constant in $k$,  summing over any sub-collection $\mathcal{B}'$ of broad cubes $B_k \in \mathcal{B}$ gives
\begin{equation}\label{eq:broad3}
    \|Ef\|_{L^p(Y)}^p \lesssim K^{O(1)} \left(\frac{M}{\# \mathcal{B}'}\right)^{O(1)} \left \|\prod_{j=1}^m |Ef_{j}|^{1/m}\right\|_{L^p(\cup_\mathcal{B'}B)}^{p}.
\end{equation}

In order to exploit the transversality of caps $\tau_j$, we apply the following multilinear refined Strichartz estimate from \cite{DGLZ18}:

\begin{thm}\label{thm:MultiRefiStri}
    Let $q = \frac{2(m+1)}{m-1}$. Suppose that $(\tau_1,...,\tau_m) \in \Gamma$ are $m$ transverse caps, where $\Gamma$ is defined as in \eqref{eq:m-transverse}. Let $Z = Z(P_1,...,P_{n+1-m})$ be a transverse complete intersection with $\deg P_i \leq R^{\delta_{\deg}}$, where $0<\delta_{\deg{}} \ll \delta \ll \varepsilon$ are small parameters. Let $f_j \in L^2$ be supported on $\tau_j$ and concentrated in wave packets from $\T_Z(E)$. 

    Let $Q_1,...,Q_N$ be lattice $R^{1/2}$-cubes in $ N _{ER^{1/2}}(Z) \cap B^{n+1}(0,R)$ such that
    $$
    \|Ef_j\|_{L^q(Q_k)} \text{ is dyadically a constant in }k, \text{ for each }j=1,2,...,m.
    $$

    Then
    $$
    \left \|\prod_{j=1}^m |Ef_{j}|^{1/m}\right\|_{L^q(\cup_{k=1}^N Q_k)} \lessapprox E^{O(1)} R^{-\frac{n+1-m}{2(m+1)}} N^{-\frac{m-1}{m(m+1)}} \prod_{j=1}^m \|f_{j}\|_{2}^{1/m}.
    $$
\end{thm}

However, the right hand side of \eqref{eq:broad3} cannot be estimated directly by Theorem \ref{thm:MultiRefiStri}:
\begin{itemize}
    \item \textit{Lebesgue exponents do not match: $p = \frac{2m}{m-2} > q = \frac{2(m+1)}{m-1}$.}\\
    The locally constant principle allows us to pass from $L^p$ to $L^q$ via the ``reverse'' H\"older's inequality:
    \begin{equation}\label{eq:broadReverseHolder}
        \left \|\prod_{j=1}^m |Ef_{j}|^{1/m}\right\|_{L^p(B)} \lesssim K^{O(1)} \left \|\prod_{j=1}^m |Ef_{j}|^{1/m}\right\|_{L^q(B)}.
    \end{equation}
    The $\ell^p$ over the collection $\mathcal{B}$ can be better estimated by $\ell^q$ if we dyadic pigeonhole for a significant sub-collection of broad cubes $B$ so that the left hand side of \eqref{eq:broadReverseHolder} is comparable for all $B$ in this sub-collection. Precisely, for a dyadic number $A \in [R^{-C}\|f\|_2 , R^C\|f\|_2]$, let $\mathcal{B}_A$ be a sub-collection of $\mathcal{B}$ such that
    $$
    \left \|\prod_{j=1}^m |Ef_{j}|^{1/m}\right\|_{L^p(B)} \sim A
    $$
    for each $B \in \mathcal{B}_A$. Here, $C$ is a large constant so that the contribution from other dyadic numbers $A$ can be ignored. There are $\sim \log R \sim_\varepsilon \log K$ many dyadic numbers $A$. Therefore, by pigeonholing, there exists some dyadic number $\Tilde{A}$ such that $\mathcal{B}_{\Tilde{A}}$ contains at least $K^{-C}M$ many cubes $B$.

    For any sub-collection $\mathcal{B}'$ of $\mathcal{B}_{\Tilde{A}}$, we have
    \begin{equation}\label{eq:broadDyadPigeon1}
        \left \|\prod_{j=1}^m |Ef_{j}|^{1/m}\right\|_{L^p(\cup_{\mathcal{B}'}B)} \lesssim |\# \mathcal{B}'|^{\frac{1}{p}-\frac{1}{q}} \left( \sum_{B \in \mathcal{B}'} \left \|\prod_{j=1}^m |Ef_{j}|^{1/m}\right\|_{L^p(B)}^q\right )^{1/q}.
    \end{equation}
    By \eqref{eq:broadReverseHolder} and \eqref{eq:broadDyadPigeon1}, we have
    \begin{equation}\label{eq:broadDP+RH}
        \left \|\prod_{j=1}^m |Ef_{j}|^{1/m}\right\|_{L^p(\cup_{\mathcal{B}'}B)} \lesssim K^{O(1)}|\# \mathcal{B}'|^{\frac{1}{p}-\frac{1}{q}} \left \|\prod_{j=1}^m |Ef_{j}|^{1/m}\right\|_{L^q(\cup_{\mathcal{B}'}B)}.
    \end{equation}
    \item \textit{The arrangement of $B\in \mathcal{B}$ in each lattice $R^{1/2}$-cubes does not obey the assumption of Theorem \ref{thm:MultiRefiStri}.}\\
    We further sort the collection $\mathcal{B}_{\Tilde{A}}$ as follows. For a dyadic number $\lambda \in [1, R^{\frac{n+1}{2}}]$ and dyadic numbers $\iota_1,...,\iota_m \in [R^{-C}\|f\|_{2}, R^C\|f\|_{2}]$, let $\mathcal{B}_{\Tilde{A},\lambda,\iota_1,...,\iota_m}$ be a sub-collection of $\mathcal{B}_{\Tilde{A}}$ such that each lattice $R^{1/2}$-cube $Q$ containing some cubes $B$ in the sub-collection contains $\sim \lambda$ cubes from $\mathcal{B}_{\Tilde{A}}$, and
    $$
    \|E f_{j}\|_{L^q(Q)} \sim \iota_{j}, \quad \text{for each }j=1,...,m.
    $$
    Similarly, there are $O_\varepsilon(\log K)$ many dyadic numbers $\lambda, \iota_1,...,\iota_m$. By pigeonholing, there exists a sub-collection $\mathcal{B}' = \mathcal{B}_{\Tilde{A},\lambda,\iota_1,...,\iota_m}$ of $\mathcal{B}_{\Tilde{A}}$ which contains at least $ K^{-C}M$ many cubes $B$. Let $Q_1,...,Q_N$ be lattice $R^{1/2}$-cubes containing some $B \in \mathcal{B}'$. By counting the number of cubes $B \in \mathcal{B}'$, we have the relation
    \begin{equation}
        K^{-C}M \lesssim \lambda N. 
    \end{equation}
    By Theorem \ref{thm:MultiRefiStri} and the relation above, we obtain
    \begin{equation}\label{eq:broadMRS}
        \left \|\prod_{j=1}^m |Ef_{j}|^{1/m}\right\|_{L^q(\cup_{\mathcal{B}'}B)} \lesssim_\varepsilon K^{O(1)} 
        E^{O(1)} R^{-\frac{n+1-m}{2(m+1)}+\frac{\varepsilon}{2}} \left(\frac{M}{\lambda}\right)^{-\frac{m-1}{m(m+1)}} \prod_{j=1}^m \|f_{j}\|_{2}^{1/m}.
    \end{equation}
\end{itemize}

Combining \eqref{eq:broad3}, \eqref{eq:broadDP+RH}, \eqref{eq:broadMRS} and the fact that $\|f_{j}\|_{2} \leq \|f\|_{2}$, we get
\begin{equation}\label{eq:broad4}
    \|Ef\|_{L^p(Y)} \lesssim_{\varepsilon} K^{O(1)} E^{O(1)} R^{-\frac{n+1-m}{2(m+1)}+\frac{\varepsilon}{2}} M^{\frac{1}{p}-\frac{1}{q}}\left(\frac{M}{\lambda}\right)^{-\frac{m-1}{m(m+1)}} \|f\|_{2}.
\end{equation}
Moreover, by putting $r=R^{1/2}$ into \eqref{eq:prop:induct:ndimassu}, we have
    \begin{equation}\label{eq:broadLambda}
        \lambda \leq \gamma R^{n/2},
    \end{equation}
and by considering $B(x',r) = 2B_k$ in \eqref{eq:prop:induct:ndimassu}, we obtain
\begin{equation}\label{eq:broadGamma}
    1 \leq \gamma (2K^2)^n.
\end{equation}

On the other hand, by a theorem of Wongkew \cite{Wongkew} on 
the volume of neighborhoods of real algebraic varieties, $N_{ER^{1/2}}(Z)\cap B^{n+1}(0,R)$ can be covered by $\lesssim E^{n+1-m}R^{m/2} (R^{\delta_{\deg}})^{n+1-m}$ many balls of radius $R^{1/2}$. Each $R^{1/2}$-ball contain at most $\gamma R^{n/2}$ many cubes $B_k$. Therefore,
\begin{equation}\label{eq:broadM}
    M \leq  E^{O(1)} \gamma R^{\frac{m+n}{2}+O(\delta_{\deg})}.
\end{equation}

To prove \eqref{eq:prop:induct}, by \eqref{eq:broad4} and \eqref{eq:broadLambda}, it suffices to show that
$$
R^{-\frac{n+1-m}{2(m+1)}} M^{\frac{1}{p}-\frac{1}{q}}\left(\frac{M}{\gamma R^{n/2}}\right)^{-\frac{m-1}{m(m+1)}} \leq E^{O(1)} K^{O(1)}\left ( \frac{\gamma}{M} \right)^{\frac{1}{p_m} - \frac{1}{p}} R^{\frac{n+m}{2}(\frac{1}{p_m}-\frac{1}{2}) + \frac{m}{2(m+1)}+\frac{\varepsilon}{2}}.
$$
Rearranging gives
$$
M^{\frac{m-2}{m^3-m}} \leq E^{O(1)} K^{O(1)} \gamma^{\frac{2m-3}{m^3-m}} R^{\frac{(m-2)(m+n)}{2(m^3-m)}+\frac{\varepsilon}{2}},
$$
which can be seen by \eqref{eq:broadGamma} and \eqref{eq:broadM}.

We have finished the proof of the broad case.

\subsection{Narrow case} \label{sec:narrow}

Recall from the definition that for each narrow $K^2$-cube $B$, there is an $(m-1)$-dimensional subspace $V$ such that for all $\tau \in \S(B)$,
$$
\angle(G(\tau),V) \leq \frac{1}{100nK}.
$$
The narrow assumption allows us to apply the following $(m-1)$-dimensional decoupling inequality:
\begin{align}\label{eq:lem:dec_variety}
     &\|Ef\|_{L^p(B)} \sim \left \|\sum_{\tau \in \S(B)} E f_\tau \right\|_{L^p(B)} \notag\\
        &\lesssim_\varepsilon K^{\varepsilon^4} \left( \sum_{\tau \in \mathcal{S(B)}} \|Ef_\tau\|_{L^p(10B)}^2\right)^{1/2} + R^{-1000n}\|f\|_{2}\\
         &\leq K^{\varepsilon^4} \left( \sum_{\tau \in \mathcal{S}} \|Ef_\tau\|_{L^p(10B)}^2\right)^{1/2} + R^{-1000n}\|f\|_{2}. \notag
\end{align}
where $p = \frac{2m}{m-2}$, and $\mathcal{S}$ is the set of $K^{-1}$-cubes tiling $B^{n}(0,1)$.

Fix $\tau \in \mathcal{S}$. We hope to bound $\|Ef_{\tau}\|_{L^p(10B)}$ on the right-hand side of \eqref{eq:lem:dec_variety} by the induction on radius following a parabolic rescaling argument. For each $1/K$-cube $\tau$ in $B^n(0,1)$, we write $\xi=\xi_0+K^{-1} \zeta \in \tau$, where $\xi_0$ is the center of $\tau$. Then 
$$
|E f_{\tau} (x',x_{n+1})| =K^{-n/2} |Eg (\tilde x', \tilde x_{n+1})|
$$
for some function $g$ with Fourier support in the unit cube and $\|g\|_2=\|f_{\tau}\|_2$, where the new coordinates $(\tilde x', \tilde x_{n+1})$ are related to the old coordinates $(x',x_{n+1})$ by
\begin{equation} \label{coord}
\begin{cases}
   \tilde x' =K^{-1} x' + 2 x_{n+1} K^{-1} \xi_0\,, \\
   \tilde x_{n+1} = K^{-2} x_{n+1} \,.
\end{cases}
\end{equation}
For simplicity, denote the above relation by $(\tilde x',\tilde x_{n+1})=F(x',x_{n+1})$. Denote $R_1\leq R$ the new scale. Note that the preimage of $R_1$-cubes under the map $F$ are boxes $\square$ of dimensions $R_1 K \times ... R_1 K \times R_1 K^2$ pointing to the direction given by $(-2\xi_0,1)$. A natural choice of $R_1$ is such that $R_1K^2 = R$. We tile $B^{n+1}(0,R)$ by these boxes $\square_{\tau,D}$. By a smooth partition of unity, we can write $f_\tau = \sum_{D} f_{\square_{\tau,D}}$, where $Ef_{\square_{\tau,D}}$ is essentially supported on $\square_{\tau,D}$, with a Schwartz tail. For reading convenience, we neglect the Schwartz tail in the sequel.

Now we have $f=\sum_{\tau}f_\tau=\sum_{\tau}\sum_{D} f_{\square_{\tau,D}}$. To ease the notations, we simply write $f = \sum_{\square} f_\square$. Since each $K^2$-cube $B$ lies in exactly one box $\square_{\tau,D}$, the decoupling inequality \eqref{eq:lem:dec_variety} implies that 
\begin{equation}\label{eq:narrow_dec}
    \|Ef\|_{L^p(B)} \lesssim_\varepsilon K^{\varepsilon^4} \left( \sum_{\square} \|Ef_\square\|_{L^p(10B)}^2\right)^{1/2},
\end{equation}
where the tail term $R^{-1000n}\|f\|_{2}$ is neglected.

Recall that $K = R^{\delta}$ and $R_1 = R/K^2$ is the new scale. Set $K_1  = R_1^\delta$. The preimages of $K_1^2$-cubes $\Tilde{B}_k$ under the parabolic rescaling $F$ are tubes $S$ of dimensions $K_1^2K \times .. \times K_1^2 K \times K_1^2K^2$ pointing to the same direction as $\square$. Therefore, we tile $\square$ by these tubes $S$. Figure \ref{fig:parabolic_rescaling} below shows various important objects in the physical space under the parabolic rescaling.

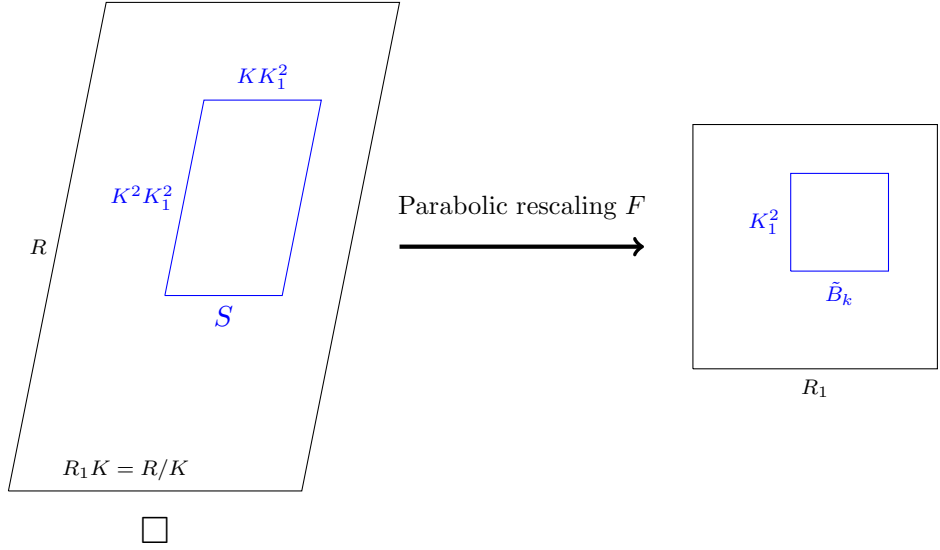
\begin{figure}[ht] 
    \centering
    \begin{tikzpicture}[scale = 1.3]
        \draw (0,0) -- (3,0) node[pos=0.4, above] {\tiny $R_1K = R/K$}--(3,0) -- (4,5) -- (1,5) -- (0,0) node[pos=0.5, left] {\tiny $R$};
        %\draw[red] (0,0) -- (2,0) -- (4,5) -- (2,5) node[pos=0.5, above] {$T$} --(0,0);
        \draw[blue] (1.6,2) -- (2.8,2) node[pos=0.5, below] {\small $S$} to (3.2,4) -- (2,4)node[pos=0.5, above] {\tiny $K K_1^2$}   -- (1.6,2) node[pos=0.5, left] {\tiny $K^2K_1^2$};
        %\fill[color=black!50] (2.1,3.0) -- (2.1,3.6)  -- (2.7, 3.6) -- (2.7,3.0) -- (2.1,3.0) node[pos=0.5, below , color = black] {\tiny $B$};
        \draw[ultra thick, ->] (4,2.5) to (6.5,2.5);
        \node at (1.5,-0.4) {\large $\square$};
        
        \node at (5.25,2.9) {\footnotesize Parabolic rescaling $F$};

        \draw (7,1.25) -- (9.5,1.25) node[pos=.5,below] {\tiny $R_1$} to (9.5,3.75) to (7,3.75) -- (7,1.25);
        %\draw[red] (7,1.25) -- (26/3,1.25) -- (19/2,3.75) -- (47/6,3.75) node[pos=0.5, above] {$\Tilde{T}$} --(7,1.25);
        \draw[blue] (8,2.25) -- (9,2.25)node[pos=0.5, below] {\tiny $\tilde{B}_k$} -- (9,3.25) -- (8,3.25)  -- (8,2.25) node[pos=0.5, left] {\tiny $K_1^2$};
        %\fill[color=black!50] (8.25,2.75) -- (8.15,3.05)  -- (8.65, 3.05) -- (8.75,2.75) -- (8.25,2.75);
    \end{tikzpicture}
    \caption{important objects in the physical space under the parabolic rescaling}
    \label{fig:parabolic_rescaling}
\end{figure}

Moreover, wavepackets $T$ of dimensions $R^{1/2+\delta}\times ... \times R^{1/2+\delta} \times R $ are mapped to wavepackets $\Tilde{T}$ at scale $R_1$ of dimensions $R_1^{1/2+\delta+O(\delta^2)}\times ... \times R_1^{1/2+\delta+O(\delta^2)} \times R_1 $ under $F$. It is straightforward to check that the tangent-to-variety condition is preserved under parabolic rescaling: $T$ is $ER^{-1/2}$-tangent to $\Z(P_1,\cdots,P_{n+1-m})$ implies that the rescaled tube $\Tilde{T}$ is ${ER_1^{-1/2}}$-tangent to $Z(Q_1,\cdots,Q_{n+1-m})$, where $Q_j = P_j \circ F^{-1}$. For example, see the proof of Theorem 7.1 in \cite{DGL17}.

In what follows, we will perform dyadic pigeonholing so that each rescaled $f_\square$, which we denote by $\Tilde{f}_\square$, satisfies the assumptions in the Proposition \ref{prop:induct} at scale $R_1$ with dyadic parameters $\gamma_1$, $M_1$ uniformly in $\square$. To sum over all $B \in Y$ efficiently, we will also sort $B$ and $\square$ so that the rescaled $B$ is contained in $\sim \mu$ many $K_1^2$ cubes and 
$\|f_\square\|_2$ is dyadically a constant in $\square$.

In the following list, we first write down the desired \textit{properties} by performing the dyadic pigeonholing in \textit{italic} font, followed by how we sort narrow $B \subset Y$, $S \subset \square$ or $\square$ to achieve these \textit{goals}. 
\begin{enumerate}
    \item \textit{$\|E\Tilde{f}_\square\|_{L^p(\Tilde{B}_k)}$ is dyadically a constant.} Since $\Tilde{f}_\square$ and $\Tilde{B}_k$ are the rescaling of $f_\square$ and $S$ respectively, it suffices to have $\|Ef_\square\|_{L^p(S)}$ dyadically a constant. We sort $S \subset \square$ according to the value of $\|Ef_\square\|_{L^p(S)}$. For a dyadic number $\beta_1$, let $\mathbb{S}_{\square,\beta_1}$ be the collection of $S$ such that $\|Ef_\square\|_{L^p(S)} \sim \beta_1$.
    \item \textit{Each $S$ contains $\sim \eta$ narrow $K^2$-cubes $B$.} We further sort $S \in \mathbb{S}_{\square,\beta_1}$ according to the number of narrow $K^2$-cubes it contains. For a dyadic number $\eta$, let $\mathbb{S}_{\square,\beta_1,\eta}$ be the sub-collection of $\mathbb{S}_{\square,\beta_1}$ such that each $S$ contains $\sim \eta$ narrow $K^2$-cubes $B$. Denote $Y_{\square,\beta_1,\eta}$ the union of tubes $S$ in $\mathbb{S}_{\square,\beta_1,\eta}$.
    \item \textit{Parameters $M_1$, $\gamma_1$ and the value of $\|f_\square\|_2$ are dyadically constants in $\square$.} Since $F^{-1}$ respectively maps $K_1^2$-balls $\Tilde{B}_k$ and $r$-balls $B_r$ to $S$ and tubes $T_r$ of dimensions $Kr \times ...\times Kr \times K^2r$ pointing to the same direction as $\square$, we sort the boxes $\square$ according to the number of $S \in \mathbb{S}_{\square,\beta_1,\eta}$, the value of
    \begin{equation}\label{eq:narrow_gamma_1}
        \max_{T_r \subset \square: r \in [K_1^2,\sqrt{R_1}]} \frac{\# \{ S \in \mathbb{S}_{\square,\beta_1,\eta} : S \subset T_r\}}{r^n}
    \end{equation}
    and $\|f_\square\|_2$. For dyadic numbers $M_1, \gamma_1, \beta_2$, let $\mathbb{B}_{\beta_1,\eta,M_1, \gamma_1, \beta_2}$ be the collection of boxes $\square$ such that 
    $$
    \# \mathbb{S}_{\square,\beta_1,\eta} \sim M_1,\quad \text{the quantity in } \eqref{eq:narrow_gamma_1} \sim \gamma_1, \quad \|f_\square\|_2 \sim \beta_2.
    $$
    After rescaling, for each $\square \in \mathbb{B}_{\beta_1,\eta,M_1, \gamma_1, \beta_2}$, $\Tilde{f}_\square$ and $K_1^2$-cubes in $F(Y_{\square,\beta_1,\eta})$ satisfy the assumptions in Proposition \ref{prop:induct} at scale $R_1$ with parameters $\gamma_1,M_1$.
    \item \textit{Each narrow $K^2$-cube $B$ is contained in $Y_{\square, \beta_1,\eta}$ for $\sim \mu$ many $\square \in \mathbb{B}_{\beta_1,\eta,M_1, \gamma_1, \beta_2}$.} For a dyadic number $\mu$, let $Y'_{\beta_1,\eta,M_1, \gamma_1, \beta_2, \mu}$ be the union of narrow $K^2$-cubes $B$ such that 
    \begin{equation}\label{eq:narrow_mu}
        \#\{\square \in \mathbb{B}_{\beta_1,\eta,M_1, \gamma_1, \beta_2} : B \subset Y_{\square, \beta_1,\eta}\} \sim \mu.
    \end{equation}
\end{enumerate}
    By construction, $\sqcup_{\beta_1,\eta} Y_{\square, \beta_1,\eta} = \square$.  Therefore, on $10B$, we have
    $$
    Ef_\square =  \sum_{\beta_1,\eta} Ef_{\square} \cdot \chi_{Y_{\square, \beta_1,\eta}} 1_{10B \subset Y_{\square, \beta_1,\eta }},
    $$
    where $\chi_{Y_{\square, \beta_1,\eta}}$ is the characteristic function on $Y_{\square, \beta_1,\eta}$ and $1_{10B \subset Y_{\square, \beta_1,\eta }}$ is the indicator function on the condition $10B \subset Y_{\square, \beta_1,\eta} $ . Here, we use the fact that the dimensions of $B$ is much smaller than that of $S$.

    By the triangle inequality and construction that $\mathbb{B}_{\beta_1,\eta,M_1, \gamma_1, \beta_2}$ are disjoint collections of $\square$, \eqref{eq:narrow_dec} implies that
    \begin{equation}\label{eq:narrow_dya1}
        \|Ef\|_{L^p(B)} \lesssim_\varepsilon K^{\varepsilon^4} \sum_{\beta_1,\eta,M_1, \gamma_1, \beta_2} \left( \sum_{\substack{\square \in \mathbb{B}_{\beta_1,\eta,M_1, \gamma_1, \beta_2}  \\10B \subset Y_{\square, \beta_1,\eta }}} \|Ef_\square\|_{L^p(10B)}^2\right)^{1/2},
    \end{equation}
    for each narrow $K^2$-cubes $B$.

    The total contributions from the sets where $\beta_1\leq R^{-C}\|f\|_2$ or $\beta_2\leq R^{-C}\|f\|_2$ are negligible for some large enough $C$. Therefore, we can bound the interesting dyadic parameters by
    $$
    R^{-C}\|f\|_2 \leq \beta_1, \beta_2 \leq R^C \|f\|_2 , \quad 1\leq M_1,\mu,\eta \leq R^{O(1)}, \quad K^{-2n}\leq \gamma_1 \leq R^{O(1)}.
    $$

    There are $O(\log R)$ many dyadic numbers in these range. For each narrow $K^2$-cubes $B$, we choose parameters $\beta_1(B),\eta(B), M_1(B) ,\gamma_1(B) ,\beta_2(B)$ that maximize the term inside the parentheses in \eqref{eq:narrow_dya1}. We further pick the most popular parameters $\beta_1,\eta,M_1,\gamma_1,\beta_2$ among all narrow $K^2$-cubes $B$ and $\mu$ with the largest $|Y'_{\beta_1,\eta,M_1, \gamma_1, \beta_2, \mu}|$. With these parameters, we have \textit{properties (1) - (4)} and the decoupling inequality
    \begin{equation}\label{eq:narrow_dya2}
        \|Ef\|_{L^p(B)} \lesssim_\varepsilon K^{\varepsilon^4} (\log R)^5 \left( \sum_{\substack{\square \in \mathbb{B}_{\beta_1,\eta,M_1, \gamma_1, \beta_2}  \\10B \subset Y_{\square, \beta_1,\eta }}} \|Ef_\square\|_{L^p(10B)}^2\right)^{1/2},
    \end{equation}
    for $\gtrsim (\log R)^{-6} M$ many narrow $K^2$-cubes $B \subset Y'_{\beta_1,\eta,M_1, \gamma_1, \beta_2, \mu}$. 
    
    We are done with the dyadic pigeonholing argument and fix parameters $\beta_1,\eta,M_1,\gamma_1,\beta_2, \mu$ for the rest of the proof. Abbreviate $\mathbb{B}_{\beta_1,\eta,M_1, \gamma_1, \beta_2}$ , $Y_{\square, \beta_1,\eta }$ and $Y'_{\beta_1,\eta,M_1, \gamma_1, \beta_2, \mu}$ by $\mathbb{B}$, $Y_\square$ and $Y'$ respectively. 
    
    We are now ready to apply parabolic rescaling $F$ and the induction hypothesis. First, 
    \begin{equation}\label{eq:narrow_ParaRes}
        \frac{\|Ef_{\square}\|_{L^p(Y_\square)}}{\|f_\square\|_2} = K^{\frac{n+2}{p} - \frac{n}{2}}\frac{\|E\Tilde{f}_\square\|_{L^p(F(Y_\square))}}{\|\tilde{f}_\square\|_2}.
    \end{equation}
    Recall that the tangent-to-variety condition is preserved under parabolic rescaling. Together with \textit{properties (1) and (3)},  $\Tilde{f}_\square$ and  $F(Y_\square)$ satisfy the assumptions in Proposition \ref{prop:induct} at scale $R_1$ with parameters $\gamma_1 ,M_1$. By induction on radius $R$, we obtain
    \begin{equation}\label{eq:narrow_indu}
        \|E\Tilde{f}_\square\|_{L^p(F(Y_\square))} \lesssim_\varepsilon E^C R_1^\varepsilon \left ( \frac{\gamma_1}{M_1} \right)^{\frac{1}{p_m} - \frac{1}{p}} R_1^{\frac{n+m}{2}(\frac{1}{p_m}-\frac{1}{2}) + \frac{m}{2(m+1)}} \|\Tilde{f}_\square\|_2.
    \end{equation}
    
    We are ready to estimate $\|Ef\|_{L^p(Y)}$:
    \begin{align*}
        \|Ef\|_{L^p(Y)} &\lesssim (\log R)^{O(1)} \left (\sum_{B \subset Y'}\|Ef\|_{L^p(B)}^p \right)^{1/p}\notag\\
        &\lesssim K^{\varepsilon^4} (\log R)^{O(1)} \mu^{\frac{1}{2}-\frac{1}{p}} \left (\sum_{ \square \in \mathbb{B} : 10B \subset Y_\square}\|Ef_\square\|_{L^p(10B)}^p \right)^{1/p} \notag\\
        &\lesssim K^{\varepsilon^4} (\log R)^{O(1)} \mu^{\frac{1}{2}-\frac{1}{p}}K^{\frac{n+2}{p}-\frac{n}{2}} E^C R_1^{\varepsilon} \left( \frac{\gamma_1}{M_1} \right)^{\frac{1}{p_m}- \frac{1}{p}} R_1^{\frac{n+m}{2}(\frac{1}{p_m}-\frac{1}{2}) + \frac{m}{2(m+1)}} \left ( \sum_{\square \in \mathbb{B}}\|f_\square\|_2^p\right)^{1/p} \notag\\
        &\lesssim K^{\frac{n+2}{p}-\frac{n}{2}+ 2\varepsilon^4} E^C \left(\frac{\mu}{\# \mathbb B}\right)^{\frac{1}{2}-\frac{1}{p}} \left( \frac{\gamma_1}{M_1} \right)^{\frac{1}{p_m}- \frac{1}{p}} \left ( \frac{R}{K^2}\right)^{\frac{n+m}{2}(\frac{1}{p_m}-\frac{1}{2}) + \frac{m}{2(m+1)}+\varepsilon} \|f\|_2^2.
    \end{align*}
    The first inequality follows from the assumption that $\|Ef\|_{L^p(B_k)}$ is essentially a constant and that $Y'$ contains $\gtrsim (\log R)^{-6} M$ many narrow $K^2$-cubes $B_k$. The second inequality follows from \textit{property (4)}, \eqref{eq:narrow_dya2} and H\"older's inequality. The third inequality follows from parabolic rescaling \eqref{eq:narrow_ParaRes} and the induction hypothesis \eqref{eq:narrow_indu}. The last inequality follows from $(\log R)^{O(1)}\lesssim_\varepsilon K^{\varepsilon^4}$ and \textit{property (3)} that $\|f_\square\|_2$ is dyadically a constant.

    We claim that 
    \begin{equation}\label{eq:narrow_claim}
        \frac{\mu}{\# \mathbb B} \lesssim \min \left \{ 1, K^{\varepsilon^{4}} \frac{M_1 \gamma K^{n+1}}{ M \gamma_1} \right\}.
    \end{equation}

    Assume the claim for now. Since $1/2 \ge 1/p_m$, we estimate
    \begin{align*}
        \left(\frac{\mu}{\# \mathbb B}\right)^{\frac{1}{2}-\frac{1}{p}} \left( \frac{\gamma_1}{M_1} \right)^{\frac{1}{p_m}- \frac{1}{p}}  
        &\lesssim K^{O(\varepsilon^{4})} \left (\frac{M_1 \gamma K^{n+1}}{ M \gamma_1} \right)^{\frac{1}{p_m}-\frac{1}{p}}\left( \frac{\gamma_1}{M_1} \right)^{\frac{1}{p_m}- \frac{1}{p}} \\
        &= K^{O(\varepsilon^{4}) + (n+1) \left(\frac{1}{p_m}- \frac{1}{p}\right)}\left( \frac{\gamma}{M} \right)^{\frac{1}{p_m}- \frac{1}{p}}.
    \end{align*}
    
    We now check the power of $K$. Recall that $p=\frac{2m}{m=1}$ and $p_m=\left( \frac{1}{2} - \frac{1}{m^3-m} \right)^{-1}$. A direct computation shows that
    $$
    \frac{n+2}{p} - \frac{n}{2} +(n+1) \left(\frac{1}{p_m}- \frac{1}{p}\right) -2\left[\left (\frac{n+m}{2}\right)\left(\frac{1}{p_m}-\frac{1}{2}\right) + \frac{m}{2(m+1)}\right] = 0.
    $$
    Therefore, the induction for the narrow case closes as long as $\frac{K^{O(\varepsilon^4)}}{K^{2\varepsilon}} \ll 1$, which can be achieved by taking $K=R^\delta$ sufficiently large compared to any constant depending on $\varepsilon$.

    It remains to prove the claim \eqref{eq:narrow_claim}.

    Recall from the definition of $\mu$ in \eqref{eq:narrow_mu} that for each $B\subset Y'$,
    $$
    \#\{\square \in \mathbb{B} : B \subset Y_{\square}\} \sim \mu.
    $$
    The first upper bound $ \frac{\mu}{\# \mathbb B} \leq 1$ is obvious. To see the second bound, we first consider the cardinality of the set $\{(\square,B): \square \in \mathbb B , B \subset Y_\square \cap Y'\}$ in two ways. On one hand, there are at least $(\log R)^{-6} M $ many $K^2$-cubes $B$ in $Y'$, each of which is contained in about $\mu$ many $Y_\square$. We obtain a lower bound
    $$
    \# \{(\square,B): \square \in \mathbb B , B \subset Y_\square \cap Y'\} \gtrsim (\log R)^{-6}M\mu.
    $$
    On the other hand, for each $\square \in \mathbb B$, $Y_\square$ contains about $M_1$ many tubes $S$ by \textit{property (3)} and each $S$ contains about $ \eta$ many narrow $K^2$-cubes $B$ by \textit{property (2)}. Since not all the narrow $K^2$-cubes are in $Y'$, we have only an upper bound
    $$
    \# \{(\square,B): \square \in \mathbb B , B \subset Y_\square \cap Y'\} \lesssim (\# \mathbb B) M_1 \eta.
    $$
    Combining the estimates, we have
    \begin{equation}\label{eq:narrow.geo1}
        \frac{\mu}{\# \mathbb B} \lesssim \frac{(\log R)^6 M_1 \eta}{M}.
    \end{equation}

    Let $T_r \subset \square$ be the maximizer in \eqref{eq:narrow_gamma_1}, for some $r \in [K_1^2,\sqrt{R_1}]$. We count the cardinality of $\{ B \subset Y: B \subset T_r\}$ in two ways. On one hand, $T_r$ contains about $\gamma_1 r^{n}$ many $S$ by \textit{property (3)} and each $S$ contains about $ \eta$ many narrow $K^2$-cubes $B$ by \textit{property (2)}. Since not all $K^2$-cubes $B\subset Y$ are narrow, we have only a lower bound
    $$
    \# \{ B \subset Y: B \subset T_r\}\gtrsim \gamma_1 r^n \eta.
    $$

    On the other hand, $T_r$, a tube of dimensions $Kr \times...\times Kr \times K^2 r $ can be partitioned into $K$ many $Kr$-balls. Since $Kr \in [K^2 , \sqrt{R}]$, by assumption, each $Kr$-balls contains at most $\gamma (Kr)^n$ many $K^2$-cubes $B$ in $Y$. Thus, we obtain an upper bound
    $$
    \# \{ B \subset Y: B \subset T_r\}\lesssim K\gamma (Kr)^n .
    $$
    Combining the estimates, we have
    \begin{equation}\label{eq:narrow.geo2}
        \eta \lesssim \frac{\gamma K^{n+1}}{\gamma_1}.
    \end{equation}

    The second bound in \eqref{eq:narrow_claim} follows from \eqref{eq:narrow.geo1}, \eqref{eq:narrow.geo2} and the estimate $(\log R)^6 \lesssim_\varepsilon K^{\varepsilon^4}$. We have proved the claim, and hence the narrow case.
\end{proof}

\bibliography{myreference}

\begin{bibdiv}
\begin{biblist}

\bib{Bo12}{article}{
      author={Bourgain, J.},
       title={On the {S}chr\"{o}dinger maximal function in higher dimension},
        date={2013},
        ISSN={0371-9685},
     journal={Tr. Mat. Inst. Steklova},
      volume={280},
      number={Ortogonal\cprime nye Ryady, Teoriya Priblizheni\u{\i} i Smezhnye Voprosy},
       pages={53\ndash 66},
         url={https://doi.org/10.1134/s0081543813010045},
      review={\MR{3241836}},
}

\bib{BD2015}{article}{
      author={Bourgain, J.},
      author={Demeter, C.},
       title={The proof of the {$l^2$} decoupling conjecture},
        date={2015},
        ISSN={0003-486X},
     journal={Ann. of Math. (2)},
      volume={182},
      number={1},
       pages={351\ndash 389},
         url={https://doi.org/10.4007/annals.2015.182.1.9},
      review={\MR{3374964}},
}

\bib{BG11}{article}{
      author={Bourgain, J.},
      author={Guth, L.},
       title={Bounds on oscillatory integral operators based on multilinear estimates},
        date={2011},
        ISSN={1016-443X},
     journal={Geom. Funct. Anal.},
      volume={21},
      number={6},
       pages={1239\ndash 1295},
         url={https://doi.org/10.1007/s00039-011-0140-9},
      review={\MR{2860188}},
}

\bib{du2019upper}{article}{
      author={Du, X.},
       title={Upper bounds for {F}ourier decay rates of fractal measures},
        date={2020},
     journal={J. Lond. Math. Soc.},
      volume={102},
      number={3},
       pages={1318\ndash 1336},
}

\bib{DGL17}{article}{
      author={Du, X.},
      author={Guth, L.},
      author={Li, X.},
       title={A sharp {S}chr\"{o}dinger maximal estimate in {$\mathbb R^2$}},
        date={2017},
        ISSN={0003-486X},
     journal={Ann. of Math. (2)},
      volume={186},
      number={2},
       pages={607\ndash 640},
         url={https://doi.org/10.4007/annals.2017.186.2.5},
      review={\MR{3702674}},
}

\bib{DGLZ18}{article}{
      author={Du, X.},
      author={Guth, L.},
      author={Li, X.},
      author={Zhang, R.},
       title={Pointwise convergence of {S}chr\"{o}dinger solutions and multilinear refined {S}trichartz estimates},
        date={2018},
     journal={Forum Math. Sigma},
      volume={6},
       pages={Paper No. e14, 18},
         url={https://doi.org/10.1017/fms.2018.11},
      review={\MR{3842310}},
}

\bib{DGOWWZ21}{article}{
      author={Du, X.},
      author={Guth, L.},
      author={Ou, Y.},
      author={Wang, H.},
      author={Wilson, B.},
      author={Zhang, R.},
       title={Weighted restriction estimates and application to {F}alconer distance set problem},
        date={2021},
        ISSN={0002-9327},
     journal={Amer. J. Math.},
      volume={143},
      number={1},
       pages={175\ndash 211},
         url={https://doi.org/10.1353/ajm.2021.0005},
      review={\MR{4201782}},
}

\bib{DKWZ20}{article}{
      author={Du, X.},
      author={Kim, J.},
      author={Wang, H.},
      author={Zhang, R.},
       title={Lower bounds for estimates of the {S}chr\"{o}dinger maximal function},
        date={2020},
        ISSN={1073-2780},
     journal={Math. Res. Lett.},
      volume={27},
      number={3},
       pages={687\ndash 692},
         url={https://doi.org/10.4310/MRL.2020.v27.n3.a4},
      review={\MR{4216564}},
}

\bib{DLschrodinger}{article}{
      author={Du, X.},
      author={Li, J.},
       title={{$L^p$} estimates of the maximal {S}chr\"odinger operator in {$\R^n$}},
        date={2024},
     journal={arXiv preprint},
}

\bib{DOKZFalconerDec}{article}{
      author={Du, X.},
      author={Ou, Y.},
      author={Ren, K.},
      author={Zhang, R.},
       title={Weighted refined decoupling estimates and application to {F}alconer distance set problem},
        date={2023},
     journal={arXiv preprint arXiv:2309.04501},
}

\bib{DZ19}{article}{
      author={Du, X.},
      author={Zhang, R.},
       title={Sharp {$L^2$} estimates of the {S}chr\"{o}dinger maximal function in higher dimensions},
        date={2019},
        ISSN={0003-486X},
     journal={Ann. of Math. (2)},
      volume={189},
      number={3},
       pages={837\ndash 861},
         url={https://doi.org/10.4007/annals.2019.189.3.4},
      review={\MR{3961084}},
}

\bib{Gu16}{article}{
      author={Guth, L.},
       title={A restriction estimate using polynomial partitioning},
        date={2016},
        ISSN={0894-0347},
     journal={J. Amer. Math. Soc.},
      volume={29},
      number={2},
       pages={371\ndash 413},
         url={https://doi.org/10.1090/jams827},
      review={\MR{3454378}},
}

\bib{Gu18}{article}{
      author={Guth, L.},
       title={Restriction estimates using polynomial partitioning {II}},
        date={2018},
        ISSN={0001-5962},
     journal={Acta Math.},
      volume={221},
      number={1},
       pages={81\ndash 142},
         url={https://doi.org/10.4310/ACTA.2018.v221.n1.a3},
      review={\MR{3877019}},
}

\bib{Wongkew}{article}{
      author={Wongkew, R.},
       title={Volumes of tubular neighbourhoods of real algebraic varieties},
        date={1993},
     journal={Pacific J. Math.},
      volume={159},
      number={1},
       pages={177\ndash 184},
}

\end{biblist}
\end{bibdiv}

\vspace{0.25cm}
	
\noindent Xiumin Du, Northwestern University, \textit{xdu@northwestern.edu}\\

\noindent Jianhui Li, Northwestern University, \textit{jianhui.li@northwestern.edu}\\

\noindent Hong Wang, NYU Courant, \textit{hw3639@nyu.edu} \\

\noindent Ruixiang Zhang, UC Berkeley, \textit{ruixiang@berkeley.edu}

\end{document}